\documentclass[twoside, a4paper, reqno, 11pt]{amsart} 
\usepackage{microtype} 
\microtypesetup{protrusion=true, expansion=true}
\usepackage{amsmath, amsfonts, amsthm, amssymb, verbatim, xcolor, multirow, booktabs, mathdots, bm, amscd, latexsym, mathrsfs, comment, mathtools}
\usepackage[pdftex, 
            hidelinks,
            pdfencoding=auto,
            psdextra]{hyperref}
\hypersetup{
    colorlinks=true,
    linkcolor=blue, 
    citecolor=blue,
    urlcolor=blue,
    hypertexnames=false,
    pdftitle={Orthogonality of Defect Ideals for Jordan Homomorphisms},
    pdfauthor={Cao Minh Nam} 
}

\usepackage[T1]{fontenc}
\usepackage{lmodern}

\usepackage[shortlabels]{enumitem}
\usepackage{aliascnt}
\usepackage[capitalize]{cleveref}
\usepackage{xspace}
\crefname{equation}{}{}

\usepackage{empheq}
\usepackage{soul}

\theoremstyle{plain}

\newtheorem{theorem}{Theorem}[section]
\newaliascnt{proposition}{theorem}
\newtheorem{proposition}[proposition]{Proposition}
\aliascntresetthe{proposition}
\newaliascnt{lemma}{theorem}
\newtheorem{lemma}[lemma]{Lemma}
\aliascntresetthe{lemma}
\newaliascnt{corollary}{theorem}
\newtheorem{corollary}[corollary]{Corollary}
\aliascntresetthe{corollary}

\newtheorem*{maintheoremA}{Theorem A}
\newtheorem*{maintheoremB}{Theorem B}
\newtheorem*{maintheoremC}{Theorem C}
\newcommand{\thmAref}{\hyperlink{thmA}{Theorem~A}\xspace}
\newcommand{\thmBref}{\hyperlink{thmB}{Theorem~B}\xspace}
\newcommand{\thmCref}{\hyperlink{thmC}{Theorem~C}\xspace}

\theoremstyle{remark}
\newtheorem{remark}{Remark}

\crefname{theorem}{theorem}{theorems}
\Crefname{theorem}{Theorem}{Theorems}

\crefname{proposition}{proposition}{propositions}
\Crefname{proposition}{Proposition}{Propositions}

\crefname{lemma}{lemma}{lemmas}
\Crefname{lemma}{Lemma}{Lemmas}

\crefname{corollary}{corollary}{corollaries}
\Crefname{corollary}{Corollary}{Corollaries}

\newcommand{\Alg}{\operatorname{Alg}}
\newcommand{\Ann}{\operatorname{Ann}}

\title[Orthogonality of Defect Ideals]
{Orthogonality of Defect Ideals for Jordan Homomorphisms}

\author{Cao Minh Nam}

\email{caominhnam.dhsp@gmail.com}

\begin{document}

\begin{abstract}
Let $\varphi\colon A\to B$ be a Jordan homomorphism of associative algebras 
over a field $\mathbb{F}$ of characteristic different from $2$, with 
defect ideals $V_{\varphi}$ and $W_{\varphi}$. If
$A=\Alg([[A,A],[A,A]]),$
then $V_{\varphi}W_{\varphi}=W_{\varphi}V_{\varphi}=0$ and
$V_{\varphi}\cap W_{\varphi}\subseteq \Ann(\Alg(\varphi(A))).$
Moreover, the induced map modulo $V_{\varphi}\cap W_{\varphi}$ is standard. 
The annihilator bound is sharp, even in finite dimension. As an application, when $A$ is unital and $A=K(A)$, every Jordan homomorphism defined on $A$ is splittable and standard. In this setting, the corresponding result of Bre\v{s}ar and Zelmanov holds without the surjectivity hypothesis.
\end{abstract}

\maketitle

\section{Introduction}

Throughout this paper $\mathbb{F}$ denotes a field of characteristic different from~$2$. Unless explicitly stated otherwise, all algebras under consideration, in particular $A$ and $B$, are associative $\mathbb{F}$-algebras possibly without identity and algebra homomorphisms are not assumed to preserve the identity. We use the convention $x\circ y=xy+yx$ for the Jordan product, as in \cite{BresarZelmanov2026}. Over fields of characteristic different from $2$, this differs only by a scalar factor from the normalized convention $x\circ y=\frac12(xy+yx)$ as in \cite{McCrimmon2004}, for example.

An $\mathbb{F}$-linear map $\varphi\colon A\to B$ is called a \emph{Jordan homomorphism} if
$$\varphi(x\circ y)=\varphi(x)\circ\varphi(y) \text{ for all } x,y\in A.$$
Thus, a Jordan homomorphism preserves only the symmetrized product and need not respect the associative multiplication in either order. A classical structural problem is to determine when the Jordan identity implies that a Jordan homomorphism decomposes as the sum of a homomorphism and an antihomomorphism. The main result of this paper gives an ideal-theoretic separation: if $A=\Alg([[A,A],[A,A]])$, then the two defect ideals measuring the failures of multiplicativity and antimultiplicativity annihilate each other.

Classical contributions to this structural problem are due to Jacobson and Rickart \cite{JacobsonRickart1950}, Herstein \cite{Herstein1956}, and Smiley \cite{Smiley1957}. Further developments are surveyed by Bre\v{s}ar and Zelmanov in \cite{BresarZelmanovSurvey2025}. We now fix the notation needed to formulate it precisely. For $x,y\in A$, write $[x,y]=xy-yx$. If $U,V\subseteq A$ are $\mathbb{F}$-subspaces, let $[U,V]$ denote the $\mathbb{F}$-span of the commutators $[x,y]$ with $x \in U$ and $y \in V$, and let $U^2$ be the $\mathbb{F}$-span of the products $xx'$ with $x,x'\in U$. For a subset $L \subseteq A$, write
$\Alg(L)=\operatorname{span}_{\mathbb{F}}\{x_1\cdots x_n \mid n\ge1, x_1,\ldots, x_n\in L\}$
for the $\mathbb{F}$-subalgebra generated by $L$, with the convention that no identity element is adjoined unless it already lies in this subalgebra. Denote by $K(A)$ the two-sided ideal generated by $[A,A]$ and by $\Ann(A)=\{x\in A\mid xA=Ax=0\}$ the two-sided annihilator.

In the terminology of \cite{BresarZelmanov2026}, a Jordan homomorphism $\varphi\colon A\to B$ is called \emph{standard} if
$\varphi=\varphi_++\varphi_-$, where $\varphi_+$ is a homomorphism and $\varphi_-$ is an antihomomorphism such that $\varphi_+(A)\varphi_-(A)=\varphi_-(A)\varphi_+(A)=0.$
A Jordan homomorphism $\varphi$ is said to be \emph{annihilator-by-standard} if the induced Jordan homomorphism
$A\to \Alg(\varphi(A))/\Ann(\Alg(\varphi(A)))$
is standard. We measure the failure of multiplicativity and antimultiplicativity by the bilinear maps
$$
\begin{aligned}
a(x,y)&=\varphi(xy)-\varphi(x)\varphi(y),\\
b(x,y)&=\varphi(xy)-\varphi(y)\varphi(x).
\end{aligned}
$$
Let $V_{\varphi}$ and $W_{\varphi}$ be the two-sided ideals of $\Alg(\varphi(A))$ generated by the values of $a$ and $b$, respectively. According to \cite[Definition~3.1]{Bresar2025}, the Jordan homomorphism $\varphi$ is \emph{splittable} if
$V_{\varphi}\cap W_{\varphi}=0$.~By \cite[Theorem~3.13]{Bresar2025}, every splittable Jordan homomorphism is the sum of a homomorphism and an antihomomorphism on the commutator ideal $K(A)$. Hence the pair of defect ideals and especially their intersection is a natural invariant of the map.

Another direction of research concerns algebras generated by double commutators. Under the assumption $A=\Alg([[A,A],[A,A]])$, Bre\v{s}ar \cite[Theorem~3.1]{Bresar2022} showed that Jordan epimorphisms onto semiprime algebras are standard. For unital algebras satisfying $A=K(A)$, Bre\v{s}ar and Zelmanov first showed in \cite[Proposition~2.5]{BresarZelmanov2026} that every Jordan homomorphism is nearly-standard in their terminology, and then proved in \cite[Theorem~2.10]{BresarZelmanov2026} standardness under the additional assumption of surjectivity. Whether surjectivity is necessary was left open. Replacing the codomain $B$ by the generated subalgebra $\Alg(\varphi(A))$ does not resolve the issue, since $\varphi(A)$ need not be multiplicatively closed and hence $\varphi$ need not be surjective onto $\Alg(\varphi(A))$.

The basic theory also provides, for example, pointwise identities \cite[(2.6)--(2.7)]{BresarZelmanov2026} controlling individual multiplicative and antimultiplicative defects. The new step is to upgrade, under generation by double commutators, this elementwise control to orthogonality of the two-sided ideals generated by all defect values. Our three principal results are stated next and their proofs are given in subsequent sections.

\hypertarget{thmA}{}\begin{maintheoremA}
Let $\varphi\colon A\to B$ be a Jordan homomorphism and assume that
$A=\Alg([[A,A],[A,A]])$. Then the defect ideals are orthogonal,
$$V_{\varphi}W_{\varphi}=W_{\varphi}V_{\varphi}=0.$$
Moreover,
$V_{\varphi}+W_{\varphi}=\Alg(\varphi(A)),$
and in particular,
$$V_{\varphi}\cap W_{\varphi}\subseteq\Ann(\Alg(\varphi(A))).$$
\end{maintheoremA}

The orthogonality assertion is a central structural conclusion of \thmAref. The additional equality $V_{\varphi}+W_{\varphi}=\Alg(\varphi(A))$ already holds under the weaker hypothesis $A=K(A)$. The proof of \thmAref is given in \Cref{sec:defect-identities}, while \thmBref below shows that its annihilator conclusion is sharp even in finite dimension.

\hypertarget{thmB}{}\begin{maintheoremB}
There exist finite-dimensional algebras $A$ and $B$ and a Jordan homomorphism $\varphi\colon A\to B$ such that the following assertions hold.
    \begin{enumerate}[label=\textup{(\arabic*)}]
    \item $A=\Alg \bigl([[A,A],[A,A]] \bigr).$

    \item $V_{\varphi}\cap W_{\varphi} = \Ann\bigl(\Alg(\varphi(A))\bigr)$,~and this common two-sided ideal is one-dimensional over $\mathbb{F}$.
    \end{enumerate}
\end{maintheoremB}

Hence the annihilator term in \thmAref cannot be omitted, even in the finite-dimensional case.  We give the proof of \thmBref in \Cref{sec:sharpness}. In the construction there, both $A$ and the generated range $\Alg(\varphi(A))$ are nonunital.

A principal application concerns the unital condition $A=K(A)$. By \cite[Lemma~2.4]{BresarZelmanov2026} and properties of values of multilinear Lie polynomials, we show that if $A$ is unital, then $A=K(A)$ holds if and only if $A=\Alg([[A,A],[A,A]])$. More precisely, if $G=[[A,A],[A,A]]$, then $A=G+G^2$. Consequently, arbitrary Jordan homomorphisms from such algebras fall within the scope of \thmAref, which yields the following answer to the surjectivity question posed by Bre\v{s}ar and Zelmanov.

\hypertarget{thmC}{}\begin{maintheoremC}
Let $A$ be a unital algebra satisfying $A=K(A)$, and let $\varphi\colon A\to B$ be any Jordan homomorphism.  Then there exists a unique central idempotent $\varepsilon\in Z(\Alg(\varphi(A)))$ such that
$x\mapsto \varepsilon\varphi(x)$ is a homomorphism and
$x\mapsto (1_{\Alg(\varphi(A))}-\varepsilon)\varphi(x)$ is an antihomomorphism. Moreover,
$V_{\varphi}\cap W_{\varphi}=0$. In particular, every Jordan homomorphism from $A$ is splittable and standard.
\end{maintheoremC}

The proof of \thmCref is given in the final section. In particular, \thmCref{} shows that the surjectivity assumption in \cite[Theorem~2.10]{BresarZelmanov2026} is unnecessary.
The remaining results describe consequences and interpretations of this structure. Put $Q=V_{\varphi}\cap W_{\varphi}$, $I=\varphi^{-1}(V_{\varphi})$, and $J=\varphi^{-1}(W_{\varphi})$. Then \Cref{cor:range} gives the canonical exact sequence
$$0 \to Q \to \Alg(\varphi(A)) \to A/I\oplus(A/J)^{\mathrm{op}} \to 0,$$
with $Q\Alg(\varphi(A))=\Alg(\varphi(A))Q=0$.~Hence the generated range is an annihilator extension of its standard quotient. We provide the corresponding Hochschild-cohomological description in \Cref{prop:cocycle-obs}, where it is used to interpret the extension and not as an independent route to the main theorem. We also prove that $Q=0$ for $s$-unital algebras generated by their double commutators, together with the corresponding special universal envelope.

The proof of the orthogonality assertion in \thmAref starts from the cyclic map theorem \cite[Theorem~2.1]{Bresar2022}. The defect identities and associativity give cross relations between $a$ and $b$, from which the quadrilinear map
$P(x,y,z,w)=a(x,y)b(z,w)$ is alternating.~A multilinear argument in six variables yields vanishing on pairs of commutator entries.~A sandwich argument then extends this vanishing to the ideals generated by the defect values, and hence $V_{\varphi}W_{\varphi}=W_{\varphi}V_{\varphi}=0.$
The annihilator inclusion follows immediately from this orthogonality and the equality
$V_{\varphi}+W_{\varphi}=\Alg(\varphi(A))$.

Apart from the present section, this paper is organized as follows. Basic facts on generated ranges and defect ideals are collected in \Cref{sec:generated-ranges}. Defect identities and the proof of \thmAref appear in \Cref{sec:defect-identities}. The sharpness of the annihilator bound, including the proof of \thmBref, is established in \Cref{sec:sharpness}. Structural consequences of the quotient by $V_{\varphi}\cap W_{\varphi}$ are discussed in \Cref{sec:consequences}. The $s$-unital, simple, and universal cases are treated in \Cref{sec:sunital-universal}. The last section concerns unital algebras satisfying $A=K(A)$ and concludes with the proof of \thmCref.

\section{Generated ranges and defect ideals}\label{sec:generated-ranges}

Let $\varphi\colon A\to B$ be a Jordan homomorphism, with defect maps
$$a(x,y)=\varphi(xy)-\varphi(x)\varphi(y) \text{ and } 
b(x,y)=\varphi(xy)-\varphi(y)\varphi(x).$$
A direct check shows that $a$ and $b$ are bilinear and skew-symmetric. Let $V_{\varphi}$ and $W_{\varphi}$ be the two-sided ideals of $\Alg(\varphi(A))$ generated by the values of $a$ and $b$ on $A$, respectively, and refer to them as the \emph{defect ideals} of $\varphi$.
The defect ideals associated with a Jordan homomorphism naturally lie in the algebra generated by its image, rather than in the ambient codomain.

We start with a few consequences of $A=\Alg([[A,A],[A,A]])$ for the domain and for the generated range $\Alg(\varphi(A))$. These elementary facts will be used repeatedly. In particular, the idempotence $\Alg(\varphi(A))=\Alg(\varphi(A))^2$ will be needed later in the analysis of annihilator extensions and $s$-unitality.

\begin{lemma}\label{lem:domain-structure}
Assume that $A=\Alg \bigl([[A,A],[A,A]] \bigr)$. Then the following assertions hold.
    \begin{enumerate}[label = \textup{(\arabic*)}]
        \item $A=K(A)=\Alg \bigl( [A,A] \bigr).$ 
        \item $A$ is linearly spanned by the Jordan products $x\circ y$ with $x,y\in A$. In particular, $A= A^2$.
        \item   $\Alg(\varphi(A)) = \Alg(\varphi(A))^2$ for any Jordan homomorphism $\varphi \colon A \to B$.
    \end{enumerate}
\end{lemma}
    \begin{proof}
    Put $G=[[A,A],[A,A]]$. Since $G\subseteq[A,A]$ and $K(A)$ is the two-sided ideal generated by $[A,A]$, we have
    $$A=\Alg(G)\subseteq\Alg([A,A])\subseteq K(A)\subseteq A.$$
    Hence all inclusions are equalities, which proves part~\textup{(1)}.
    
    \medskip
    For part~\textup{(2)}, let $J$ be the subspace spanned by the elements $x\circ y$ with $x,y\in A$. The identity
    $[x,[u,v]]=(x\circ u)\circ v-(x\circ v)\circ u$ gives $[A,[A,A]]\subseteq J$. In particular, the following inclusions hold:
    $$[[A,A],[A,A]]\subseteq[A,[A,A]]\subseteq J.$$
     Now let $r\in A$ and $g\in G$. Since $G\subseteq [A,A]$, we have $[r,g]\in [A,[A,A]]\subseteq J$, and $r\circ g\in J$. Hence
    $2rg=r\circ g+[r,g]\in J.$
    Because $\operatorname{char}\mathbb{F}\neq 2$, it follows that $rg\in J$. This gives every finite product of elements of $G$ belongs to $J$. Hence $A=\Alg(G)\subseteq J$, so $A=J$ and, in particular, $A=A^2$.
    
    \medskip
     For part~\textup{(3)}, by the preceding argument, each element of $\varphi(A)$ is a finite sum of terms $\varphi(x)\circ\varphi(y)$ with $x,y\in A$, so $\varphi(A)\subseteq \Alg(\varphi(A))^2$. Since $\Alg(\varphi(A))^2$ is a two-sided ideal, we imply that $\Alg(\varphi(A))=\Alg(\varphi(A))^2$.  
    \end{proof}

The following observations require only the weaker assumption $A = K(A)$.

\begin{proposition}\label{prop:comaximal}
Let $A$ be an algebra with $A=K(A)$, and let $\varphi\colon A\to B$ be a Jordan homomorphism with defect ideals $V_{\varphi}$ and $W_{\varphi}$. Then
$$V_{\varphi}+W_{\varphi}=\Alg(\varphi(A)).$$
\end{proposition}
\begin{proof}
Put $\overline{\Alg(\varphi(A))} = \Alg(\varphi(A))/(V_{\varphi} + W_{\varphi})$. Let $\overline{\varphi}\colon A \to \overline{\Alg(\varphi(A))}$ denote the map induced by $\varphi$. Both defects of $\overline{\varphi}$ vanish. It follows that $\overline{\varphi}$ is simultaneously a homomorphism and an antihomomorphism, which forces $\overline{\varphi}([A,A])=0$. Since $\ker\overline{\varphi}$ is a two-sided ideal containing $[A,A]$, it contains $K(A)=A$, hence $\overline{\varphi}(A)=0$. Note that $\overline{\Alg(\varphi(A))}$ is generated by $\overline{\varphi}(A)$. Thus, $\overline{\Alg(\varphi(A))}=0$, and consequently $V_{\varphi} + W_{\varphi} = \Alg(\varphi(A))$.
\end{proof}

By \Cref{prop:comaximal}, the two defect ideals already generate $\Alg(\varphi(A))$, so their images become disjoint after quotienting by the intersection. This is what separates the homomorphic and antihomomorphic components below.

\begin{proposition}\label{prop:standard-mod-Q}
Assume that $A=K(A)$. Let $\varphi\colon A\to B$ be a Jordan homomorphism and set $Q=V_{\varphi}\cap W_{\varphi}$. Then the following statements hold.
    \begin{enumerate}[label=\textup{(\arabic*)}]
    \item The quotient $\Alg(\varphi(A))/Q$ decomposes as $V_{\varphi}/Q \oplus W_{\varphi}/Q$, and the products between the two summands are zero.

    \item Let $\varphi_Q\colon A\to \Alg(\varphi(A))/Q,$ $\varphi_Q(x)=\varphi(x)+Q.$ Let $\pi_{V_{\varphi}}$ and $\pi_{W_{\varphi}}$ be the corresponding projections, regarded as endomorphisms of $\Alg(\varphi(A))/Q$, and set $\varphi_{-}=\pi_{V_{\varphi}}\circ\varphi_{Q}$ and
    $\varphi_{+}=\pi_{W_{\varphi}}\circ\varphi_{Q}$. Then $\varphi_{+}$ is a homomorphism and $\varphi_{-}$ is an  antihomomorphism whose ranges  annihilate each other. In particular, $\varphi_{Q}$ is standard.
    \end{enumerate}
\end{proposition}

\begin{proof}
    For part \textup{(1)}, we deduce from \Cref{prop:comaximal} that
    $$\Alg(\varphi(A))/Q=V_{\varphi}/Q+W_{\varphi}/Q.$$ The equality $Q = V_{\varphi} \cap W_{\varphi}$ yields trivial intersection of the summands in $\Alg(\varphi(A))/Q$, and hence the sum is direct. The inclusions $V_{\varphi}W_{\varphi} \subseteq Q$ and $W_{\varphi}V_{\varphi} \subseteq Q$ further imply that the summands annihilate each other.

    \medskip

    We now turn to part \textup{(2)}.  Since $\Alg(\varphi(A))/Q=V_{\varphi}/Q\oplus W_{\varphi}/Q$ and the two summands annihilate each other, the projections $\pi_{W_{\varphi}}$ and $\pi_{V_{\varphi}}$ are $\mathbb{F}$-algebra homomorphisms.
    So $\varphi_{+},\,\varphi_{-}\colon A\to \Alg(\varphi(A))/Q$ are the two components of $\varphi_{Q}$.
    The defect of $\varphi_{Q}$ takes values in the complementary summand. Indeed, for $x,y\in A$, the two defect expressions lie in different components,
        \begin{equation*}
        \begin{aligned}
        \varphi_{Q}(xy)-\varphi_{Q}(x)\varphi_{Q}(y) &= a(x,y)+Q \in V_{\varphi}/Q,\\
        \varphi_{Q}(xy)-\varphi_{Q}(y)\varphi_{Q}(x) &= b(x,y)+Q \in W_{\varphi}/Q.
        \end{aligned}
        \end{equation*}
    Note that the two projections split the
    defects as $\varphi_{+}(xy)=\varphi_{+}(x)\varphi_{+}(y)$ and
    $\varphi_{-}(xy)=\varphi_{-}(y)\varphi_{-}(x)$. Consequently, $\varphi_{+}$ is a homomorphism and $\varphi_{-}$ is an antihomomorphism. Since $V_{\varphi}/Q$ and $W_{\varphi}/Q$ annihilate each other, $\varphi_{+}(A)\varphi_{-}(A)=\varphi_{-}(A)\varphi_{+}(A)=0$. Hence $\varphi_{Q}=\varphi_{+}+\varphi_{-}$ is the sum of maps of these two types with orthogonal images, so $\varphi_{Q}$ is standard, as required.
    \end{proof}

\section{Defect identities}\label{sec:defect-identities}

We next establish the identities describing products involving the defect ideals. We start from the cyclic map theorem \cite[Theorem~2.1]{Bresar2022}, which applies to the skew-symmetric bilinear defects $a$ and $b$ when $A$ is generated by double commutators. The resulting cyclic identities yield relative Leibniz formulas for the two defects, while associativity in the generated range leads to the corresponding cross identities. These identities give rise to an alternating quadrilinear map whose vanishing on suitable commutator entries is the crucial step toward orthogonality of the two defect ideals.

\begin{lemma}\label{lem:cyclic}
If $A=\Alg \bigl([[A,A],[A,A]] \bigr)$, then the following two formulas
\begin{align*}
 a(xy,z)+a(zx,y)+a(yz,x)&=0,\\
 b(xy,z)+b(zx,y)+b(yz,x)&=0,
\end{align*}
hold for all $x,y,z\in A$.
\end{lemma}
    \begin{proof}
    Note that both $a$ and $b$ are skew-symmetric bilinear maps. Since $\varphi$ preserves Jordan products and $\operatorname{char}\mathbb{F} \neq 2$, we have $\varphi(x^2)=\varphi(x)^2.$ Applying this to $x^2\circ x=2x^3$ gives $\varphi(x^3)=\varphi(x)^3$, and hence
    $$a(x^2,x)=b(x^2,x)=0.$$
    Therefore, we may apply \cite[Theorem~2.1]{Bresar2022} separately to $a$ and $b$. Since
    $\Alg([[A,A],[A,A]])=A,$
    the variable in this generated algebra ranges over all of $A$, and the two cyclic identities follow.
    \end{proof}

The cyclic identities can then be rewritten as relative Leibniz formulas, which express the defect at a product via left and right multiplication of defect values by elements of $\varphi(A)$.

\begin{lemma}\label{lem:Leibniz}
If $A=\Alg \bigl([[A,A],[A,A]] \bigr)$,  then the following two formulas 
    \begin{align*}
    a(y,zx)&=\varphi(x) a(y,z)+a(y,x)\varphi(z),\\
    b(y,zx)&=\varphi(z)b(y,x)+b(y,z)\varphi(x),
    \end{align*}
hold for all $x,y,z\in A$.
\end{lemma}
    \begin{proof}
    The first formula in \Cref{lem:cyclic}, together with skew-symmetry, yields
    $$a(xy,z)=a(y,zx)+a(x,yz).$$
    Equating the two expressions for $a(xy,z)$ gives
    \begin{align*}
    a(y,zx) &=a(xy,z)-a(x,yz)\\
            &=\varphi(x)\bigl[\varphi(yz)-\varphi(y)\varphi(z)\bigr]-\bigl[\varphi(xy)-\varphi(x)\varphi(y) \bigr]\varphi(z) \\
            &=\varphi(x)a(y,z)-a(x,y)\varphi(z)\\
            &=\varphi(x)a(y,z)+a(y,x)\varphi(z).
    \end{align*}
    The same reasoning applied to $b$ gives the second formula.
    \end{proof}

    The relative Leibniz identities handle the two defects separately. We use associativity on the generated range to relate mixed products involving both defects.
    The resulting cross relations give the first direct link between the two families of identities.
\begin{proposition}\label{prop:cross}
If $A=\Alg \bigl([[A,A],[A,A]] \bigr)$, then the following formulas
\begin{align*}
 b(y,z)a(w,x)&=a(w,z)b(x,y),\\
 a(x,y)b(w,z)&=b(w,x)a(y,z),
\end{align*}
are valid for all $x,y,z,w\in A$. In particular, $a(x,y)b(z,w)=b(x,y)a(z,w)$.
\end{proposition}
    \begin{proof}
 Put $X=\varphi(x)$, $Y=\varphi(y)$, and $Z=\varphi(z)$.   Apply the first formula of \Cref{lem:Leibniz} to $(xy)z$ and $x(yz)$ to obtain
\begin{align*}
a(w,(xy)z) &= ZY\,a(w,x)+Z\,a(w,y)X+a(w,z)XY+a(w,z)a(x,y),\\
a(w,x(yz)) &= YZ\,a(w,x)+a(y,z)a(w,x)+Z\,a(w,y)X+a(w,z)YX.
\end{align*}
Taken in turn, these equalities give
\[
[Z,Y]a(w,x)+a(w,z)[X,Y]+a(w,z)a(x,y)-a(y,z)a(w,x)=0.
\]
The identities $[Z,Y]=a(y,z)-b(y,z)$ and $[X,Y]=b(x,y)-a(x,y)$ imply 
cancellation of all terms involving $a$ alone. Consequently,
\begin{equation} \label{eq:prop3.31}
b(y,z)a(w,x)=a(w,z)b(x,y).
\end{equation}
The second formula of \Cref{lem:Leibniz}, applied in the same manner, gives
\begin{align*}
b(w,(xy)z) &= XY\,b(w,z)+a(x,y)b(w,z)+X\,b(w,y)Z+b(w,x)YZ,\\
b(w,x(yz)) &= XY\,b(w,z)+X\,b(w,y)Z+b(w,x)YZ+b(w,x)a(y,z),
\end{align*}
    and thereby the following identity holds: \begin{equation} \label{eq:prop3.32}
a(x,y)b(w,z)=b(w,x)a(y,z).
\end{equation}
When combined with skew-symmetry, the identities (\ref{eq:prop3.31}) and (\ref{eq:prop3.32})  now give
\begin{align*}
a(x,y)\,b(z,w) &= \phantom{-}b(z,x)\,a(y,w)\\
               &= -b(x,z)\,a(y,w)\\
               &= -a(y,z)\,b(w,x)\\
               &= \phantom{-}a(y,z)\,b(x,w)\\
               &= \phantom{-}b(x,y)\,a(z,w).
\end{align*}
    This proves the final assertion.
    \end{proof}

We shall use the standard terminology for alternating multilinear maps. If $E$ and $L$ are vector spaces over $\mathbb{F}$, a multilinear map
$$f\colon \underbrace{E \times \dots \times E}_{n \text{ times}} \to L$$
is called \emph{alternating} if $f(x_1,\ldots,x_n)=0$ whenever two arguments coincide. Since $\operatorname{char}\mathbb{F} \neq 2$, this is equivalent to skew-symmetry.

The cross relations are assembled into a single multilinear map. Together with \Cref{prop:cross}, the skew-symmetry of $a$ and $b$ then yields further symmetry relations between the two defects.

\begin{corollary}\label{cor:P-alt}
Assume that $A=\Alg \bigl([[A,A],[A,A]] \bigr)$. Then the map $P\colon A\times A\times A\times A\to B$ defined by $P(x,y,z,w)=a(x,y)b(z,w)$ is alternating quadrilinear, and the identity
$$P([x,y],z,u,v)+P([z,x],y,u,v)+P([y,z],x,u,v)=0$$
holds for all $x,y,z,u,v\in A$.
\end{corollary}
    \begin{proof}
    It is clear that $P$ is quadrilinear. Since $a$ and $b$ are skew-symmetric, $P$ changes sign upon swapping the first two variables or the last two. For the middle pair, \Cref{prop:cross} gives
    $$P(x,z,y,w)=a(x,z)b(y,w)=b(y,x)a(z,w)=-P(x,y,z,w).$$
    Every adjacent transposition changes the sign of $P$, so $P$ is skew-symmetric and, because $\operatorname{char}\mathbb{F}\neq 2$, alternating. For the second assertion, subtract the first cyclic identity in \Cref{lem:cyclic} for $(y,x,z)$ from the same identity for $(x,y,z)$. By bilinearity, we obtain
    $$a([x,y],z)+a([z,x],y)+a([y,z],x)=0.$$
    Right multiplication by $b(u,v)$ now yields the required identity for $P$.
    \end{proof}

The following lemma isolates the purely multilinear part of the argument needed to treat two commutator entries. It is independent of the defect maps and uses only alternation and the displayed commutator identity.

\begin{lemma}\label{lem:six}
Let $R$ be an associative $\mathbb{F}$-algebra, $E$ an $\mathbb{F}$-vector space, and $\tau \colon R \times R \times R \times R \to E$ an alternating quadrilinear map. Suppose that
\begin{equation*}
\tau([x,y],z,u,v)+\tau([z,x],y,u,v)+\tau([y,z],x,u,v)=0  \eqno{\textup{($*$)}}
\end{equation*}
for all $x,y,z,u,v\in R$. Then
$$\tau([x_1,x_2],[x_3,x_4],x_5,x_6)=0$$
for all $x_1,\ldots,x_6\in R$.
\end{lemma}
\begin{proof}
Fix $x_1,\ldots,x_6\in R$ and put $c_{ij}=[x_i,x_j]$. For $i<j$ and $k<l$ with $\{i,j\}\cap\{k,l\}=\varnothing$, let $m<n$ be the complementary indices and set
$$q_{ij\mid kl}=\tau(c_{ij},c_{kl},x_m,x_n).$$
Alternation gives
\begin{equation}\label{eq:q-skew}
q_{kl\mid ij}=-q_{ij\mid kl}.
\end{equation}
Fixing the pair $12$, apply \textup{($*$)} successively to the triples $(x_3,x_4,x_5)$, $(x_3,x_4,x_6)$, $(x_3,x_5,x_6)$ and $(x_4,x_5,x_6)$, with $c_{12}$ and the remaining variable in the last two slots. 
After reordering by alternation, we obtain
\begin{align*}
q_{12\mid34}-q_{12\mid35}+q_{12\mid45}&=0,\\
-q_{12\mid34}-q_{12\mid36}+q_{12\mid46}&=0,\\
-q_{12\mid35}+q_{12\mid36}+q_{12\mid56}&=0,\\
-q_{12\mid45}+q_{12\mid46}-q_{12\mid56}&=0.
\end{align*}
Eliminating $q_{12\mid35}$, $q_{12\mid45}$, $q_{12\mid46}$ gives $2(q_{12\mid34}-q_{12\mid56})=0$. Since $\operatorname{char}\mathbb{F} \neq 2$,
\begin{equation}\label{eq:pair-12}
q_{12\mid34}=q_{12\mid56}.
\end{equation}
Relabeling the six variables gives also
$q_{34\mid12}=q_{34\mid56}$ and $q_{56\mid12}=q_{56\mid34}$. If $q=q_{12\mid34}$, then \eqref{eq:pair-12} and \eqref{eq:q-skew} give
$$q=q_{12\mid56}=-q_{56\mid12}=-q_{56\mid34}=q_{34\mid56}=q_{34\mid12}=-q,$$
so $q=0$. This is the desired identity.
\end{proof}

\begin{corollary}\label{cor:two-comm}
Assume that $A=\Alg \bigl([[A,A],[A,A]] \bigr)$. Then the formula $$a(c,x)b(d,y)=b(c,x)a(d,y)=0$$ holds for all $c,d\in[A,A]$ and all $x,y\in A$.
\end{corollary}
    \begin{proof}
    In view of \Cref{lem:six}, the cyclic identity required there is precisely the one established for the quadrilinear map $P$ in \Cref{cor:P-alt}. Hence $P(c_0,d_0,x,y)=0$ for all commutators $c_0,d_0$ in $A$ and all $x,y\in A$. Since $P$ is alternating, interchanging the two middle arguments gives $P(c_0,x,d_0,y)=0$. By bilinearity in the commutator entries, this vanishing extends to all $c,d\in[A,A]$, hence $a(c,x)b(d,y)=0$. According to the last identity in \Cref{prop:cross}, $b(c,x)a(d,y)=a(c,x)b(d,y)=0$, as required. 
    \end{proof}

Let $A$ be an algebra and write $A^{\sharp}$ for its unitization \cite[Section~2.3]{Bresar2014}. For $L\subseteq A$, the two-sided ideal generated by $L$ is $\mathbb{F}$-spanned by $Tx S,$  where $x \in L$ and  $T,S\in A^\sharp$. Thus, to pass from vanishing on the individual defect generators to vanishing on the ideals they generate, we must control products with arbitrary elements inserted between generators. The following lemma provides precisely this sandwich relation.
\begin{lemma}\label{lem:sandwich-comm}
Assume that $A=\Alg \bigl([[A,A],[A,A]] \bigr)$. Then the formula
\begin{equation*}
b(c,x)Ta(d,y)=a(c,x)Tb(d,y)=0  \eqno{\textup{($**$)}}
\end{equation*}
holds for all $c,d\in [A,A]$, $x,y\in A$ and $T\in \Alg(\varphi(A))^{\sharp}$.
\end{lemma}
    \begin{proof}
     Since $\Alg(\varphi(A))^{\sharp}$ is spanned by $1$ together with the nonempty words in $\varphi(A)$, linearity reduces the proof to these elements. We induct on word length, with $1$ as the word of length $0$. By \Cref{cor:two-comm}, \textup{($**$)} holds in case where $T=1$. Assume that $T$ can be written as a word $SZ$ of positive length, where $Z=\varphi(z)$ and $S$ is shorter. By the induction hypothesis, \textup{($**$)} holds with $S$ in place of $T$. By \Cref{lem:Leibniz},
    $Z a(d,y) = a(d,yz) - a(d,z) \varphi(y).$
    ~Hence
    $$\begin{aligned}
    b(c,x) SZ a(d,y) = b(c,x) S a(d,yz) - b(c,x) S a(d,z) \varphi(y) 
    = 0
    \end{aligned}$$
    by the induction hypothesis. Similarly, \Cref{lem:Leibniz} also gives
    $$Z b(d,y) = b(d,zy) - b(d,z) \varphi(y),$$
    and therefore $a(c,x) S Z b(d,y) = 0$ again by the induction hypothesis.
    \end{proof}

To extend the identities from commutator entries to the full defect ideals, we show that each defect ideal is generated by defect values having a commutator in one argument.

\begin{proposition}\label{prop:defect-gen}
    Assume that $A=\Alg \bigl([[A,A],[A,A]] \bigr)$. Then the following statements hold.
\begin{enumerate}[label=\textup{(\arabic*)}]
    \item The defect ideal $V_{\varphi}$ is generated by all elements of the form $a(c,x)$, where $c\in[A,A]$ and $x\in A$.
    \item The defect ideal $W_{\varphi}$ is generated by all elements of the form $b(c,x)$, where $c\in[A,A]$ and $x\in A$.
\end{enumerate}
\end{proposition}
    \begin{proof}
     Let $I$ denote the two-sided ideal generated by the elements $a(c,x)$ in which $c\in[A,A]$ and $x\in A$. Clearly, $I\subseteq V_{\varphi}$.

    By \Cref{lem:domain-structure}, $A = \Alg([A,A])$.
    Hence every element of $A$ is a linear combination of words $c_1 \cdots c_n$ in which $c_i \in [A,A]$.
    By bilinearity in the second variable, it suffices to show $a(x,c_1 \cdots c_n) \in I$ for arbitrary $x \in A$. We argue by induction on $n$. In case where $n=1$, 
    $a(x,c_1)=-a(c_1,x)\in I.$
    Assume $n\ge2$ and put $r=c_1\cdots c_{n-1}$. The first formula in \Cref{lem:Leibniz} gives
    $$a(x,rc_n)=\varphi(c_n)a(x,r)+a(x,c_n)\varphi(r).$$
    The first term lies in $I$ by the induction hypothesis, while the second does so because $a(x,c_n) = -a(c_n,x)$ is, up to sign, one of the generators. Hence every defect value $a(x,y)$ lies in $I$, so $V_{\varphi} \subseteq I$. Thus, $I = V_{\varphi}$, which proves~\textup{(1)}.

    \medskip
    Part \textup{(2)} follows by the same argument. If $J$ is the ideal generated by the $b(c,x)$, then the second formula in \Cref{lem:Leibniz} gives, for $r = c_1 \cdots c_{n-1}$,
    $$b(x, r c_n) = \varphi(r) b(x,c_n) + b(x,r) \varphi(c_n),
    $$
    and the same induction using $b(x,c_n) = -b(c_n,x)$ yields $J = W_{\varphi}$.
    \end{proof}

We now prove the main structural theorem from the Introduction.

\begin{proof}[Proof of \thmAref]
 According to \Cref{prop:defect-gen}, $V_{\varphi}$ is generated as a two-sided ideal of $\Alg(\varphi(A))$ by the elements $a(c,x)$, and $W_{\varphi}$ by the elements $b(d,y)$, where $c,d\in[A,A]$ and $x,y\in A$. Hence it suffices to consider, for $T,S,T',S' \in \Alg(\varphi(A))^{\sharp}$,
$$\bigl(Ta(c,x)S\bigr)\bigl(T'b(d,y)S'\bigr)=T\,a(c,x)(ST')b(d,y)\,S',$$
where the equality holds by associativity. Since $ST' \in \Alg(\varphi(A))^{\sharp}$,  it follows from \Cref{lem:sandwich-comm}  that $a(c,x)(ST')b(d,y)=0$, hence the product is~$0$. Interchanging $a$ and $b$ gives  $0$ for the reverse product as well. Therefore,
$$V_{\varphi}W_{\varphi}=W_{\varphi}V_{\varphi}=0.$$
By \Cref{lem:domain-structure}, the hypothesis gives $A = K(A)$, hence \Cref{prop:comaximal}~yields
$\Alg(\varphi(A)) = V_{\varphi} + W_{\varphi}.$
Let $q\in V_{\varphi}\cap W_{\varphi}$ be arbitrary and write $T\in\Alg(\varphi(A))$ as $T=v+w$, where $v\in V_{\varphi}$, $w\in W_{\varphi}$. Then the orthogonality of the two defect ideals established above yields $qT=Tq=0$, and hence $q\in\Ann(\Alg(\varphi(A)))$. Consequently, $V_{\varphi}\cap W_{\varphi}\subseteq\Ann(\Alg(\varphi(A)))$, which completes the proof.
\end{proof}

\section{Sharpness of the annihilator bound}\label{sec:sharpness}

In general, the annihilator conclusion in \thmAref cannot be strengthened to
$V_{\varphi}\cap W_{\varphi}=0.$
We now prove \thmBref\ by constructing finite-dimensional algebras and a Jordan homomorphism between them that attains the bound in \thmAref, so sharpness already occurs in finite dimension. The construction is a finite-dimensional version of the annihilator-extension method of Bre\v{s}ar and Zelmanov \cite[Proposition~3.10, Lemma~3.11 and Proposition~3.12]{BresarZelmanov2026}. The method starts from a $2$-cocycle, passes to an antisymmetric cocycle on $A\oplus A^{\mathrm{op}}$, and then forms a one-dimensional annihilator extension carrying the diagonal Jordan map. We use the same construction for a different purpose, choosing $A$ finite-dimensional with $A=\Alg([[A,A],[A,A]])$ so that the common defect equals $V_{\varphi}\cap W_{\varphi}=\Ann(\Alg(\varphi(A)))$, both sides one-dimensional.

\smallskip
Let us fix matrix notation. Let $R$ be an associative $\mathbb{F}$-algebra. For $r\in R$ and $1\le i,j\le n$, denote by $E_{ij}(r)$ the matrix with $r$ in the $(i,j)$-entry and $0$ elsewhere. If $R$ is unital with identity $1$, we write $E_{ij}$ for $E_{ij}(1)$.

\begin{proof}[Proof of \thmBref]
    Let $\Gamma =\{\omega,\mathbf a,\mathbf b,e\}$ be the semigroup consisting of four elements whose multiplication is determined by $e^2=e$, $e\mathbf a=\mathbf a$, $\mathbf b e=\mathbf b$, with all other products in $\Gamma$ equal to $\omega$. This semigroup is inspired by~\cite[Example~2.3]{LaanMarkiReimaa2018}, whose zero element is relabeled as $\omega$ here. We emphasize that $\omega$ is a basis element
    of $\mathbb F\Gamma$, not the zero vector
    of the semigroup algebra.
    
    Let $R=\mathbb{F} \Gamma $ denote the semigroup algebra with basis $\Gamma $ over $\mathbb{F}$. Then $R$ is the four-dimensional $\mathbb{F}$-vector space, 
    and its product is the $\mathbb{F}$-bilinear extension of that on $\Gamma $.   
    Associativity now follows from that of the underlying semigroup.  Moreover, by the defining relations and $e\mathbf{b}=\omega$, each basis element of $R$ is a product of two elements of $\Gamma $, so $R^2=R$. Note also that $R$ is nonunital. Indeed, any identity must be $e$, but $e\mathbf{b} = \omega \neq \mathbf{b}$.

    Define a bilinear map $f \colon R\times R\to\mathbb{F}$ by $f(\mathbf b, \mathbf a)=1$ and $f=0$ on all other ordered pairs of basis elements. With this definition, the balancing identity $$f(rs,t)=f(r,st)$$ holds for all $r,s,t\in R$. Indeed, since both sides are trilinear in $r,s,t$, it suffices to verify the identity on triples of basis elements. 
    Both $f(rs,t)$ and $f(r,st)$ vanish unless 
    $(r,s,t)=(\mathbf b,e,\mathbf a)$, corresponding to 
    $rs=\mathbf b$, $t=\mathbf a$ and $r=\mathbf b$, $st=\mathbf a$, respectively. Thus, the two sides coincide on all ordered triples of basis elements, which proves the claim. Put $A=\mathrm{M}_3(R).$ Clearly, $A$ is finite-dimensional over $\mathbb{F}$.

\medskip

\textup{(1)}
According to the hypothesis $R=R^2$, every element $r\in R$ is a finite sum $r=\sum_{\nu}u_{\nu}v_{\nu}$ with $u_{\nu}, v_{\nu} \in R$.  Let $i \neq j$. Since we are working in $\mathrm{M}_3(R)$, we can choose an index $k$ distinct from  both $i$ and $j$. Then
$$E_{ij}(r)=\sum_{\nu}[E_{ik}(u_{\nu}),E_{kj}(v_{\nu})].$$
This implies that $E_{ij}(r) \in [A,A]$. In particular, $E_{ik}(u_{\nu}), E_{kj}(v_{\nu}) \in [A,A]$. Hence
$E_{ij}(r) \in [[A,A],[A,A]].$
In addition, we also have
$$E_{ii}(r)=\sum_{\nu}E_{ij}(u_{\nu})E_{ji}(v_{\nu}).$$
Both factors in each product lie in $[[A,A],[A,A]]$. Hence $E_{ii}(r)$ belongs to the subalgebra generated by this space. Therefore, $A=\Alg([[A,A],[A,A]])$.

\medskip

\textup{(2)} We proceed in the following steps.

\smallskip
\noindent\textit{Step 1.}
 Given $X=(x_{ij}),$ $Y=(y_{ij})\in A$. To transfer the balancing identity satisfied by $f$ on $R$ to $A$,
we define $H \colon A \times A \to \mathbb{F}$ by
$$H(X,Y) = \sum_{i,j=1}^{3} f(x_{ij}, y_{ji}).$$
With $Z=(z_{ij})\in A$, the balancing identity for $f$ established above gives
$$\begin{aligned}
H(XY,Z)
&=
\sum_{i,j,k}
f(x_{ik}y_{kj},z_{ji})\\
&=
\sum_{i,j,k}
f(x_{ik},y_{kj}z_{ji})\\
&=
H(X,YZ).
\end{aligned}$$
Put $D=A\oplus A^{\mathrm{op}}$ and consider the map 
$\widetilde H\colon D\times D\to\mathbb{F}$ defined by
$$\widetilde H\bigl(X\oplus Y^{\mathrm{op}},\,X'\oplus {Y'}^{\mathrm{op}}\bigr)=H(X,X')-H(Y',Y).$$
Write $d_i = X_i \oplus Y_i^{\mathrm{op}}$. Then
$$\begin{aligned}
\widetilde{H}(d_1 d_2, d_3) &= H(X_1 X_2, X_3) - H(Y_3, Y_2 Y_1) \\
&= H(X_1, X_2 X_3) - H(Y_3 Y_2, Y_1) \\
&= \widetilde{H}(d_1, d_2 d_3),
\end{aligned}$$
which is the balancing identity for $\widetilde{H}$.
View $B=D\oplus\mathbb{F}$ as a direct sum of $\mathbb{F}$-vector spaces with multiplication
$$(d,\alpha)(d',\beta)=\bigl(dd',\widetilde H(d,d')\bigr).$$
Since $\dim_{\mathbb{F}}D<\infty$, it follows that $B$ is finite-dimensional. Now take $b_i\in B$ and write $b_i=(d_i,\alpha_i)$ with $d_i\in D$, $\alpha_i\in\mathbb{F}$. By the balancing identity for $\widetilde{H}$,
\begin{equation*}
\begin{aligned}
(b_{1}b_{2})b_{3} &= \bigl(d_{1}d_{2}d_{3},\widetilde{H}(d_{1}d_{2},d_{3})\bigr) \\
&= \bigl(d_{1}d_{2}d_{3},\widetilde{H}(d_{1},d_{2}d_{3})\bigr) \\
&= b_{1}(b_{2}b_{3}).
\end{aligned}
\end{equation*}
Therefore, the multiplication of $B$ is associative. 

\medskip
\noindent\textit{Step 2.} Identifying $B=(A\oplus A^{\mathrm{op}})\oplus\mathbb{F}$ with $A\oplus A^{\mathrm{op}}\oplus\mathbb{F}$ as vector spaces, we may write its elements as triples.
Consider the map $\varphi\colon A\to B$ defined by 
$\varphi(x)=(x,x^{\mathrm{op}},0)$ and put 
$\Delta(x,y)=H(x,y)-H(y,x)$. It is straightforward to check that
$\varphi(x)\varphi(y)=\bigl(xy,(yx)^{\mathrm{op}},\Delta(x,y)\bigr).$
Since $\Delta(y,x)=-\Delta(x,y)$, the scalar parts cancel upon symmetrization, so
$$\varphi(x)\varphi(y)+\varphi(y)\varphi(x)=\varphi(xy+yx).$$
Thus, $\varphi$ is a Jordan homomorphism, with defects
$$\begin{aligned}
a(x,y)&=\bigl(0,[x,y]^{\mathrm{op}},-\Delta(x,y)\bigr),\\
b(x,y)&=\bigl([x,y],0,\Delta(x,y)\bigr).
\end{aligned}$$
Now take
$x_0=E_{11}(\mathbf b)$,
$y_0=E_{11}(\mathbf a).$
Since $\mathbf b\mathbf a=\mathbf a\mathbf b=\omega$, we have $[x_0,y_0]=0$, while $\Delta(x_0,y_0)=f(\mathbf b,\mathbf a)-f(\mathbf a,\mathbf b)=1$.
Consequently, $$b(x_0,y_0)=(0,0,1)=-a(x_0,y_0).$$
Therefore, the element $(0,0,1)$ lies in $V_{\varphi}\cap W_{\varphi}$, so 
$0\oplus0\oplus\mathbb{F}\subseteq V_{\varphi}\cap W_{\varphi}$.

\medskip
\noindent\textit{Step 3.}
Put $C=\Alg(\varphi(A))$. Consider the canonical projection $\pi\colon B\to D$, $(d,\alpha)\mapsto d$. The two defect identities in \textit{Step 2} show that the elements 
$([x,y],0,0)$ and $(0,[x,y]^{\mathrm{op}},0)$ belong to $C$. Hence 
$[A,A]\oplus[A,A]^{\mathrm{op}}\subseteq\pi(C).$ 
As a consequence of \Cref{lem:domain-structure}, the opposite algebra $A^{\mathrm{op}}$ is generated by $[A,A]^{\mathrm{op}}.$
 It follows that  $\pi(C)$ contains both $A\oplus 0$ and $0\oplus A^{\mathrm{op}}$, hence $$\pi(C)=D.$$
From this, we imply that  $\pi(V_{\varphi})$ is the two-sided ideal of $D$ generated by elements $\pi(a(x,y))=0\oplus[x,y]^{\mathrm{op}}$ with $x,y\in A$. Thus, $\pi(V_\varphi)=0\oplus K(A)^{\mathrm{op}}$, and similarly we also have
$\pi(W_\varphi)=K(A)\oplus0$. 

By part~(1) together with 
\Cref{lem:domain-structure}, $K(A)=A$, so $\pi(V_\varphi)=0\oplus 
A^{\mathrm{op}}$ and $\pi(W_\varphi)=A\oplus0$. Moreover, $\pi^{-1}(\pi(V_\varphi))=V_\varphi+\ker\pi$, and similarly 
for $W_\varphi$. As shown in \emph{Step 2}, $\ker\pi=0\oplus0\oplus\mathbb{F}\subseteq V_{\varphi}\cap W_{\varphi}$, so $V_{\varphi}=0\oplus A^{\mathrm{op}}\oplus\mathbb{F}$ and $W_{\varphi}=A\oplus0\oplus\mathbb{F}$. In particular, $V_\varphi\cap 
W_\varphi=0\oplus0\oplus\mathbb{F}$.

\medskip
\noindent\textit{Step 4.}
Let $r=\alpha\omega+\beta\mathbf{a}+\gamma\mathbf{b}+\delta e\in
\Ann(R)$. Since  $er=re=0$, we get $r=0$, so $\Ann(R)=0$. Next, let 
$X=(x_{ij})\in\Ann(A)$. For $r\in R$, multiplying $X$ by elementary matrices on the right and 
left gives $x_{ij}r=rx_{ij}=0$ for all $i,j$, so 
$x_{ij}\in\Ann(R)=0$. This yields $X=0$, and thereby $\Ann(A)=0$. Moreover, since 
$D=A\oplus A^{\mathrm{op}}$, we have $\Ann(D)=0$.
Now take $(d,\alpha)\in\Ann(B)$. Multiplying by $(d',0)$ on the left and 
right and comparing the $D$-component, we obtain $d\in\Ann(D)$. Thus, $d=0$, and so 
$$\Ann(B)=0\oplus0\oplus
\mathbb{F}.$$
It remains to identify the generated range.
We have shown that the map~$\pi$ defined in \textit{Step 3} satisfies 
$\pi(C)=D$. Hence for each $d\in D$, there exists
$\alpha\in\mathbb{F}$ such that $(d,\alpha)\in C$.  Thus, $C$ contains
$D\oplus0$, and thereby $C=B.$
Combining this with the description of the defect ideals obtained
above, we imply that
$V_\varphi\cap W_\varphi
=
\Ann(C)
=
0\oplus0\oplus\mathbb{F}
\neq 0.$
In particular,
\[
\dim_{\mathbb{F}}(V_\varphi\cap W_\varphi)=1,
\]
which proves part~\textup{(2)}.
    \end{proof}

\begin{remark}\label{rem:sharpness}
By \thmBref, even when
$A = \Alg([[A,A],[A,A]]),$
the defect intersection can coincide with the entire annihilator of the algebra generated by the image. Hence the annihilator term in
\thmAref cannot be omitted in general. In particular, $A = \Alg([[A,A],[A,A]])$ alone does not force a Jordan homomorphism
to be splittable.
\end{remark}

\section{Consequences for the defect intersection}\label{sec:consequences}

 According to \thmAref, $V_{\varphi}\cap W_{\varphi}\subseteq \Ann(\Alg(\varphi(A))),$ and \Cref{prop:standard-mod-Q} shows that the Jordan homomorphism induced modulo $V_{\varphi}\cap W_{\varphi}$ is standard. Consequently, this decomposition descends to the annihilator quotient, and hence we obtain the following.

\begin{corollary}\label{cor:ann-standard}
Assume that $A=\Alg \bigl([[A,A],[A,A]] \bigr)$. Then any Jordan homomorphism $\varphi\colon A\to B$ is annihilator-by-standard.
\end{corollary}

\begin{proof}
Put $N=\Ann(\Alg(\varphi(A)))$ and $Q=V_{\varphi} \cap W_{\varphi}$. According to  \thmAref, $Q\subseteq N$, so that
passing to quotients yields a canonical algebra homomorphism
$$\rho \colon \Alg(\varphi(A))/Q \twoheadrightarrow \Alg(\varphi(A))/N.$$
On the other hand, \Cref{lem:domain-structure} gives $A=K(A)$. Hence  the induced Jordan homomorphism $\varphi_Q\colon A \to \Alg(\varphi(A))/Q$ is standard by \Cref{prop:standard-mod-Q}. Namely, $\varphi_{Q}=\varphi_{+}+\varphi_{-}$, where $\varphi_{+}$ is a homomorphism, $\varphi_{-}$ is an antihomomorphism and $\varphi_{+}(A)\varphi_{-}(A)=\varphi_{-}(A)\varphi_{+}(A)=0$. So $\rho\circ\varphi_{+}$ and $\rho\circ\varphi_{-}$ retain the same type and satisfy $(\rho\circ\varphi_{+})(A)(\rho\circ\varphi_{-})(A)=(\rho\circ\varphi_{-})(A)(\rho\circ\varphi_{+})(A)=0$.
Consequently, $\rho \circ \varphi_{Q}\colon A\to \Alg(\varphi(A))/N$ is standard. As it is precisely the Jordan homomorphism induced by $\varphi_{N} \colon A\to \Alg(\varphi(A))/N$ with $\varphi_N(x)=\varphi(x) +N$, the desired conclusion follows.
\end{proof}

By \thmAref, the intersection of the defects lies in the annihilator of the generated range. The results below can be compared with the classical results \cite{BaxterMartindale1979,Herstein1956,Smiley1957} on Jordan homomorphisms in the semiprime and prime~cases.
 
\begin{corollary}\label{cor:semiprime-prime}
Assume that $A=\Alg \bigl([[A,A],[A,A]] \bigr)$. Let $\varphi\colon A\to B$ be a Jordan homomorphism.
Then the following statements hold.
    \begin{enumerate}[label=\textup{(\arabic*)}]
    \item If $\Ann(\Alg(\varphi(A)))=0$, then $\varphi$ is standard. In particular, this
    holds whenever $\Alg(\varphi(A))$ is semiprime.

    \item If $\Alg(\varphi(A))$ is prime, then $\varphi$ is either a homomorphism or an
    antihomomorphism.
    \end{enumerate}
\end{corollary}
    \begin{proof}
   Put $Q=V_{\varphi} \cap W_{\varphi}$. If $\Ann(\Alg(\varphi(A)))=0$, then \thmAref implies $Q=0$, and $\varphi$ is standard by \Cref{prop:standard-mod-Q}(2). The latter statement of part~\textup{(1)} follows immediately from the semiprimeness of $\Alg(\varphi(A))$.
    
   \medskip
    For part~\textup{(2)}, we suppose that $\Alg(\varphi(A))$ is prime. \Cref{prop:standard-mod-Q}(1) gives the decomposition $\Alg(\varphi(A))=V_{\varphi}\oplus W_{\varphi}$ with $V_{\varphi}W_{\varphi}=W_{\varphi}V_{\varphi}=0$.  The primeness of $\Alg(\varphi(A))$ implies that one of the two ideals must be zero. Therefore, $\varphi$ is either a homomorphism or an antihomomorphism.
    \end{proof}

The preceding corollary gives conditions for $Q$ to vanish. Even when $Q\neq 0$, the quotient $\Alg(\varphi(A))/Q$ admits a canonical description in terms of the original algebra, refining the standard decomposition.

\begin{corollary}\label{cor:range}
Assume that $A=\Alg \bigl([[A,A],[A,A]] \bigr)$. Let $\varphi\colon A\to B$ be a Jordan homomorphism. Set $Q=V_{\varphi} \cap W_{\varphi}$,
$I=\varphi^{-1}(V_{\varphi})$, and
$J=\varphi^{-1}(W_{\varphi}).$
Then the following statements hold.
    \begin{enumerate}[label=\textup{(\arabic*)}]
    \item $I$, $J$ are two-sided ideals of $A$ and
    $I\cap J=\varphi^{-1}(Q).$
    \item There is a canonical exact sequence of $\mathbb{F}$-algebras
    $$0\to Q\to \Alg(\varphi(A))
    \xrightarrow{\;\pi\;}A/I\oplus(A/J)^{\mathrm{op}}\to 0,$$
    where
    $\pi(\varphi(x)) = \bigl(x+I,(x+J)^{\mathrm{op}}\bigr)$
    for $x\in A$.
    \item $\Ann(\Alg(\varphi(A))/Q)=\Ann(\Alg(\varphi(A)))/Q$.
    \end{enumerate}
\end{corollary}

    \begin{proof}
 By \Cref{prop:standard-mod-Q}, $\Alg(\varphi(A))/Q = V_0 \oplus W_0$ with $V_0 = V_{\varphi}/Q$ and $W_0 = W_{\varphi}/Q$. Moreover, the induced map decomposes as $\varphi_Q = \varphi_+ + \varphi_-$,
    where $\varphi_+(A) \subseteq W_0$ and $\varphi_-(A) \subseteq V_0$.

    \medskip
    For part \textup{(1)}, we observe that $\ker\varphi_{+}=\varphi^{-1}(V_{\varphi})$ and
    $\ker\varphi_{-}=\varphi^{-1}(W_{\varphi}).$
    Consequently, $I$ and $J$ are two-sided ideals of $A$, with $$I\cap J=\varphi^{-1}(V_{\varphi})\cap\varphi^{-1}(W_{\varphi})=\varphi^{-1}(Q).$$
    
    \smallskip
    For part \textup{(2)}, from the fact $\Alg(\varphi(A))/Q=\Alg(\varphi_Q(A))$, the canonical projections onto $W_{0}$ and $V_{0}$ yield $W_{0}=\Alg(\varphi_{+}(A))$ and $V_{0}=\Alg(\varphi_{-}(A))$. Since the image of an (anti)homomorphism is a subalgebra, $W_0=\varphi_{+}(A)$ and $V_0=\varphi_{-}(A)$. Then $\varphi_{+}$ and $\varphi_{-}$ induce $\mathbb{F}$-algebra isomorphisms
$\alpha\colon A/I\to W_{0}$, $\alpha(x+I)=\varphi_{+}(x)$ and
$\beta\colon (A/J)^{\mathrm{op}}\to V_{0}$, $\beta((x+J)^{\mathrm{op}})=\varphi_{-}(x)$, for $x\in A$. Define $\theta\colon \Alg(\varphi(A))/Q \to A/I\oplus (A/J)^{\mathrm{op}}$ by $v+w\mapsto \bigl(\alpha^{-1}(w),\beta^{-1}(v)\bigr)$. It is easy to verify by a routine check that $\theta(\varphi(x)+Q)=\bigl(x+I,(x+J)^{\mathrm{op}}\bigr)$.
    Now let $\pi\colon \Alg(\varphi(A))\to A/I\oplus (A/J)^{\mathrm{op}}$ be the composite of the canonical epimorphism $\Alg(\varphi(A))\to \Alg(\varphi(A))/Q$ with $\theta$. Then $\pi$ is  surjective with $\ker\pi = Q$, so the sequence
$$0 \to Q\to \Alg(\varphi(A)) \xrightarrow{\ \pi\ } A/I\oplus (A/J)^{\mathrm{op}} \to 0$$
of $\mathbb{F}$-algebras is exact and $\pi(\varphi(x))=\bigl(x+I,(x+J)^{\mathrm{op}}\bigr)$ for $x \in  A$.

    \medskip

    For part \textup{(3)}, it follows from \thmAref that $Q \subseteq \Ann(\Alg(\varphi(A)))$, and so $\Ann(\Alg(\varphi(A)))/Q \subseteq \Ann(\Alg(\varphi(A))/Q)$. Conversely, the inclusions $$T \Alg(\varphi(A))+\Alg(\varphi(A))T \subseteq Q\subseteq \Ann(\Alg(\varphi(A)))$$ hold for all $T \in \Alg(\varphi(A))$ such that $T +Q \in \Ann(\Alg(\varphi(A))/Q)$. In view of \Cref{lem:domain-structure}, we have the equality $\Alg(\varphi(A))=\Alg(\varphi(A))^2$. Thus, every $r\in\Alg(\varphi(A))$ can be written as a finite sum $r=\sum_i u_i v_i$ with $u_i,v_i\in\Alg(\varphi(A))$.
Multiplying $r$ by $T$ on the left and on the right gives $Tr=\sum_i (Tu_i)v_i=0$ and $rT=\sum_i u_i(v_i T)=0$. Hence $T\in\Ann(\Alg(\varphi(A)))$, so $\Ann(\Alg(\varphi(A))/Q)=\Ann(\Alg(\varphi(A)))/Q$, which completes the proof.
    \end{proof}

The exact sequence in \Cref{cor:range} is an annihilator extension and, in particular, a square-zero extension. We now describe this extension in terms of Hochschild cohomology.  In the final proposition of this section, we set $C=\Alg(\varphi(A))$ and $D=A/I\oplus(A/J)^{\mathrm{op}}$. Through the canonical isomorphism $D\cong C/Q$, the natural $C/Q$-bimodule structure on $Q$ gives rise to a $D$-bimodule structure defined by
$d\cdot q = cq$ and $q\cdot d = qc$,
where $c\in C$ is any lift of $d$ and $q\in Q$. This action is well-defined because $Q^{2}=0$, and it is trivial since $QC=CQ=0$. Hence an $\mathbb{F}$-bilinear map $f\colon D\times D\to Q$ is a Hochschild $2$-cocycle exactly when $f(d_1d_2,d_3)=f(d_1,d_2d_3)$ for all $d_1,d_2,d_3\in D$. For a $1$-cochain $\lambda\colon D\to Q$, the differential then takes the simple form $(\delta\lambda)(d_1,d_2)=-\lambda(d_1d_2)$.

\begin{proposition}\label{prop:cocycle-obs}
Assume that $A=\Alg \bigl([[A,A],[A,A]] \bigr)$. Let $\varphi\colon A\to B$ be a Jordan homomorphism with the notation of \Cref{cor:range}.
Consider the canonical exact sequence
$$0\to Q\to C \xrightarrow{\pi} D\to 0$$
and let $s\colon D \to C$ be an $\mathbb{F}$-linear section of $\pi$. Define $$f\colon D\times D\to Q, (d_1,d_2)\mapsto s(d_{1})s(d_{2})-s(d_{1}d_{2}).$$ Then $f$ is a Hochschild $2$-cocycle with coefficients in the trivial $D$-bimodule $Q$, and its cohomology class is independent of the chosen section~$s$.
Moreover, the following statements are equivalent.
    \begin{enumerate}[label=\textup{(\roman*)}]
    \item $Q=0$.
    \item The exact sequence above splits as a sequence of
    $\mathbb{F}$-algebras.
    \item $f$ is a Hochschild $2$-coboundary.
    \end{enumerate}
\end{proposition}

\begin{proof}
 By \thmAref, $QC=CQ=0$, and exactness gives $Q=\ker\pi$. Hence $f$ takes its values in $Q$. If we expand the associativity relation
$$\bigl(s(d_1)s(d_2)\bigr)s(d_3)=s(d_1)\bigl(s(d_2)s(d_3)\bigr)$$
and use $QC=CQ=0$, the terms involving values of $f$ vanish. This implies that $f$ satisfies $f(d_1d_2,d_3) = f(d_1,d_2d_3)$, so it is a Hochschild $2$-cocycle for the trivial $D$-bimodule $Q$.

Now let $s' = s + \lambda$ be another linear section, where $\lambda\colon D \to Q$ is an $\mathbb{F}$-linear map. Since $QC=CQ=0$, it follows that
$$f'(d_1,d_2)=f(d_1,d_2)-\lambda(d_1d_2)=f(d_1,d_2)+(\delta\lambda)(d_1,d_2).$$
Thus, $f'-f=\delta\lambda$, so the Hochschild cohomology class of $f$ does not depend on the choice of $\mathbb{F}$-linear sections. We next characterize algebra splittings in terms of the vanishing of the cohomology class of $f$. If $f$ is a Hochschild $2$-coboundary, then $f=\delta\mu$ for some $\mathbb{F}$-linear $1$-cochain $\mu\colon D\to Q$. In particular, $f(d_1,d_2)=-\mu(d_1d_2)$ for all $d_1,d_2\in D$. Replacing $s$ by the linear section
$s'=s-\mu$ gives
$f'(d_1,d_2)=f(d_1,d_2)+\mu(d_1d_2)=0,$
so $s'$ is an algebra homomorphism and the extension splits.  Conversely, if $s'=s+\lambda$ is an algebra section, then the factor set determined by it is zero, and so
$$0=f'(d_1,d_2)=f(d_1,d_2)-\lambda(d_1d_2).$$
Hence $f(d_1,d_2)=\delta(-\lambda)(d_1, d_2)$, and $f$ is a coboundary. Therefore, \textup{(ii)} and \textup{(iii)} are equivalent.
It remains to compare splitting with $Q=0$. The implication from \textup{(ii)} to \textup{(i)} uses essentially the idempotence $C=C^2$ from \Cref{lem:domain-structure}(3). Indeed, if the extension splits through an algebra section $\sigma$, then
$C=\sigma(D)\oplus Q$ as $\mathbb{F}$-vector spaces. Since $Q$ annihilates $C$, every product of two elements of $C$ lies in $\sigma(D)$, and hence
$C^2\subseteq\sigma(D).$
So $C\subseteq\sigma(D)$, which in turn implies $Q=0$. Conversely, if $Q=0$, then $\pi$ is an $\mathbb{F}$-algebra isomorphism, so its inverse provides an algebra section. Hence \textup{(i)} implies \textup{(ii)}, and the proof is now complete.
\end{proof}

\section{The \texorpdfstring{$s$}{s}-unital, simple, and universal cases}\label{sec:sunital-universal}

We now turn to cases where the annihilator can be controlled more sharply. For noncommutative simple algebras, the required generation by double commutators is known from classical Lie-structure methods. In the $s$-unital case, one-sided local identities in the quotient force the defect intersection to vanish. Throughout this section, we set $Q=V_{\varphi}\cap W_{\varphi}$.

\subsection{Simple and \texorpdfstring{$s$}{s}-unital domains}

In the noncommutative simple case, no new generation argument is needed, as the reasoning in the proof of \cite[Corollary~2.3]{Bresar2022} already yields that every noncommutative simple algebra satisfies $A=\Alg([A,A])=\Alg([[A,A],[A,A]]).$
Thus, the canonical quotient description immediately specializes as follows.

\begin{corollary}\label{cor:simple-range}
Let $\varphi\colon A\to B$ be a nonzero Jordan homomorphism from a noncommutative simple algebra. Then $Q \subseteq \Ann(\Alg(\varphi(A)))$, and $\Alg(\varphi(A))/Q$ is isomorphic to one of $A$, $A^{\mathrm{op}}$ or $A\oplus A^{\mathrm{op}}$.
\end{corollary}
\begin{proof}
In the notation of \Cref{cor:range}, simplicity of $A$ forces each of the two-sided ideals $I$ and $J$ to be either $0$ or $A$. If both coincide with $A$, then \thmAref together with that corollary yields $\varphi(A)\subseteq Q\subseteq\Ann(\Alg(\varphi(A))).$ Hence every product of two elements of $\varphi(A)$ is zero, so $\Alg(\varphi(A))^2=0$. 
On the other hand, \Cref{lem:domain-structure}(3) gives $\Alg(\varphi(A)) = \Alg(\varphi(A))^2$, which forces $\Alg(\varphi(A)) = 0$, contradicting $\varphi \neq 0$. Hence $(I,J)\neq (A,A)$.  The remaining three cases then yield the claimed quotients via \Cref{cor:range}.
\end{proof}

Following Tominaga \cite[before Lemma~1]{Tominaga1975}, an $\mathbb{F}$-algebra $R$ is said to be left (resp. right) $s$-unital if $RI=I$ (resp. $IR=I$) for every left (resp. right) ideal $I$ of $R$.~According to \cite[Theorem~1]{Tominaga1976}, this is equivalent to the finite subset formulation, namely that for every finite subset $S\subseteq R$, there is $e\in R$ such that $es=s$ (resp.\ $se=s$) for all $s\in S$.  The algebra $R$ is called \emph{$s$-unital} when it is $s$-unital on both sides. In this case, every finite subset admits a two-sided local identity. Indeed, if $e_{\ell}$ and $e_{r}$ are, respectively, a left and a right identity for a finite set $S$, then $e=e_{\ell}+e_{r}-e_{r}e_{\ell}$ satisfies $es=se=s$ for all $s\in S$. This differs from the stronger notion of a ring with local units in the sense of \'{A}nh and M\'{a}rki \cite{AnhMarki1987}, where the local identity common to a finite subset is required to be idempotent. A stronger condition appears in \cite{Abrams1983}, namely that the ring has a set of local units, i.e., a commuting set $E$ of idempotents such that for each $x\in R$, there is $e\in E$ with $ex=xe=x$.

For later use, we collect the following stability properties of one-sided $s$-unitality.

\begin{lemma}\label{lem:sunital-permanence}
Left and right $s$-unitality are preserved under quotients and direct products. Moreover, passage to the opposite algebra interchanges the two one-sided notions. In particular, the class of $s$-unital algebras is closed under quotients, opposite algebras, and direct products.
\end{lemma}

The elementary observation below eliminates the annihilator term.
 
\begin{lemma}\label{lem:sunital-lift}
Let $R$ be a ring satisfying $R=R^{2}$ and let $I$ be a two-sided ideal of $R$ with $I \subseteq\Ann(R)$. If $R/I$ is left $s$-unital, then $I=0$.
\end{lemma}

\begin{proof}
Let $q\in I$. Since $R=R^{2}$, it follows that $q=\sum_{i=1}^{n}a_{i}b_{i}$ for some $a_{i},b_{i}\in R$.
For each $x\in R$, denote by $\bar x$ its image in $R/I$. Since $R/I$ is left $s$-unital, there is $\bar e\in R/I$ with $\bar e\bar a_{i}=\bar a_{i}$ for every $i$. Then $ea_{i}-a_{i}\in I\subseteq\Ann(R)$, so $(ea_{i}-a_{i})b_{i}=0$ for every $i$. Hence $eq-q=\sum_{i}(ea_{i}-a_{i})b_{i}=0$.  Combining this with the hypothesis $I\subseteq\Ann(R)$, we obtain $q=0$ and hence $I=0$.
\end{proof}

The canonical quotient description in \Cref{cor:range} and \Cref{lem:sunital-lift} together yield a splittability criterion that requires only one-sided $s$-unitality of the relevant quotients, rather than $s$-unitality of the whole algebra.

\begin{corollary}\label{cor:map-sunital}
Assume that $A=\Alg \bigl([[A,A],[A,A]] \bigr)$. Let $\varphi\colon A\to B$ be a Jordan homomorphism. Put $I=\varphi^{-1}(V_{\varphi})$ and
$J=\varphi^{-1}(W_{\varphi}).$ Suppose that $A/I$ is left $s$-unital and $A/J$ is right $s$-unital. 
Then $Q = 0$, and in particular, $\varphi$ is splittable and standard.
\end{corollary}
\begin{proof}
By  \Cref{cor:range}, we have $\Alg(\varphi(A))/Q\cong A/I\oplus (A/J)^{\mathrm{op}}$.
\Cref{lem:sunital-permanence} shows that $(A/J)^{\mathrm{op}}$ is left $s$-unital, which together with the hypothesis that $A/I$ is left $s$-unital implies, again by \Cref{lem:sunital-permanence}, that their direct sum is left $s$-unital, and hence so is $\Alg(\varphi(A))/Q$. On the other hand, \thmAref gives $Q\subseteq\Ann(\Alg(\varphi(A)))$ and \Cref{lem:domain-structure}(3) gives $\Alg(\varphi(A))=\Alg(\varphi(A))^2$. Applying \Cref{lem:sunital-lift} to $\Alg(\varphi(A))$ and $Q$ gives $Q=0$, and the result now follows from \Cref{prop:standard-mod-Q}.
\end{proof}

The preceding criterion holds automatically when $A$ itself is $s$-unital, so the defect intersection vanishes in the following case.

\begin{theorem}\label{theo:sunital}
Let $A$ be an $s$-unital algebra with $A=\Alg \bigl([[A,A],[A,A]] \bigr)$. Then every Jordan homomorphism from $A$ is splittable and standard.
\end{theorem}
\begin{proof}
Let $\varphi\colon A\to B$ be a Jordan homomorphism. Put $I=\varphi^{-1}(V_\varphi)$ and $J=\varphi^{-1}(W_\varphi)$. By \Cref{lem:sunital-permanence}, $A/I$ and $A/J$ are $s$-unital. Hence \Cref{cor:map-sunital}  yields $Q=0$. Therefore, $\varphi$ is splittable and standard.
\end{proof}

Together with the observation on simple algebras above, the preceding theorem implies that every Jordan homomorphism from a noncommutative simple $s$-unital algebra into an associative algebra is splittable and standard.

\subsection{The special universal envelope}

We use the usual special universal envelope of a Jordan algebra. For an associative algebra $A$, let $A^{(+)}$ be the Jordan algebra with the same underlying vector space as $A$ and product $x\circ y = xy+yx$. We recall that a Jordan algebra is \emph{special} if it is isomorphic to a Jordan subalgebra of $A^{(+)}$ for some associative algebra $A$.

For a Jordan algebra $\mathcal{J}$, let $T^{+}(\mathcal{J})=\bigoplus_{n\ge1}\mathcal{J}^{\otimes n}$ be the reduced tensor algebra of the underlying vector space of $\mathcal{J}$, and let $\mathcal I_\mathcal{J}$ denote the two-sided ideal generated by the elements $x\otimes y + y\otimes x - x\circ y$ for all $x,y\in \mathcal{J}$.
The \emph{special universal envelope} of $\mathcal{J}$ is $\mathrm{U}_s(\mathcal{J})=T^{+}(\mathcal{J})/\mathcal I_\mathcal{J}$, equipped with the canonical Jordan homomorphism $u_\mathcal{J}\colon \mathcal{J}\to \mathrm{U}_s(\mathcal{J})^{(+)}$. This is Jacobson's construction \cite[\S3, pp.~514--515]{Jacobson1951}.

The pair $(\mathrm{U}_s(\mathcal{J}),u_\mathcal{J})$ is characterized by the following universal property. For any associative algebra $B$ and any Jordan homomorphism $\psi\colon \mathcal{J}\to B^{(+)}$, there exists a unique $\mathbb{F}$-algebra homomorphism $\widetilde\psi\colon \mathrm{U}_s(\mathcal{J})\to B$ satisfying $\widetilde\psi\circ u_\mathcal{J}=\psi$. Moreover, $\mathrm{U}_s(\mathcal{J})=\Alg(u_\mathcal{J}(\mathcal{J}))$ and the pair $(\mathrm{U}_s(\mathcal{J}),u_\mathcal{J})$ is determined up to a unique isomorphism compatible with the canonical maps. In particular, $u_\mathcal{J}$ is injective if and only if $\mathcal{J}$ is special.

\begin{theorem}\label{thm:universal-extension}
Let $A$ be an algebra such that
$A=\Alg \bigl([[A,A],[A,A]] \bigr),$
and let $(U,u)$ be the special universal envelope of the Jordan algebra $A^{(+)}$.~Then the following assertions hold.
\begin{enumerate}[label=\textup{(\arabic*)}]
\item $V_u\cap W_u \subseteq\Ann(U)$.
\item There is a canonical surjective homomorphism $\pi\colon U\twoheadrightarrow A\oplus A^{\mathrm{op}}$ characterized by $\pi(u(x))=(x,x^{\mathrm{op}})$ for all $x\in A$ and whose kernel is $V_u\cap W_u$.
\end{enumerate}
\end{theorem}

\begin{proof}
Since $U=\Alg(u(A))$, it follows that $u\colon A\to U$ is a Jordan homomorphism, so \thmAref gives $V_u\cap W_u\subseteq\Ann(U)$, which is part~\textup{(1)}.

\medskip

For part~\textup{(2)}, view the identity map on $A$ as a Jordan homomorphism $A^{(+)}\to A^{(+)}$. By the universal property, it extends to an algebra homomorphism
$\rho\colon U\to A$ with $\rho(u(x))=x.$
Similarly, viewing $x\mapsto x^{\mathrm{op}}$ as the canonical Jordan homomorphism $A^{(+)}\to (A^{\mathrm{op}})^{(+)}$, we obtain an algebra homomorphism $\rho^{\mathrm{op}}\colon U\to A^{\mathrm{op}}$ with $\rho^{\mathrm{op}}(u(x))=x^{\mathrm{op}}$.
Since $\rho(a_u(x,y))=0$ and $\rho^{\mathrm{op}}(b_u(x,y))=0$, we have $V_u\subseteq\ker\rho$ and $W_u\subseteq\ker\rho^{\mathrm{op}}$. Hence $u^{-1}(V_u)=u^{-1}(W_u)=0$. \Cref{cor:range} yields a canonical surjection $\pi\colon U\twoheadrightarrow A\oplus A^{\mathrm{op}}$, $\pi(u(x))=(x,x^{\mathrm{op}})$ for $x\in A$, with $\ker\pi=V_u\cap W_u$.
\end{proof}

For $s$-unital $A$, the universal defect intersection vanishes by \Cref{theo:sunital}. Hence the annihilator extension in \Cref{thm:universal-extension} is trivial and the universal envelope coincides with $A\oplus A^{\mathrm{op}}$.

\begin{corollary}\label{cor-universal}
Let $A$ be an $s$-unital algebra with $A=\Alg\bigl([[A,A],[A,A]]\bigr).$
Then $\mathrm{U}_s(A^{(+)})$ is canonically isomorphic to $A\oplus A^{\mathrm{op}}$ via the $\mathbb{F}$-algebra isomorphism $\pi \colon \mathrm{U}_s(A^{(+)})\to A\oplus A^{\mathrm{op}}$ with $\pi (u_{A^{(+)}}(x))=(x,x^{\mathrm{op}})$ for $x\in A$.
\end{corollary}

\section{The unital case \texorpdfstring{$A=K(A)$}{A=K(A)}}

We turn to the unital case $A=K(A)$. Under this hypothesis, standardness for Jordan epimorphisms was proved in \cite{BresarZelmanov2026}, where the necessity of surjectivity was left open. We show that, for unital algebras, the condition $A=K(A)$ is equivalent to
$A=\Alg\bigl([[A,A],[A,A]]\bigr).$
Thus, arbitrary Jordan homomorphisms from such algebras fall within the scope of \thmAref.

Let $\mathcal{X}=\{x_1,x_2,\ldots\}$ be a countably infinite set of noncommuting variables, $\mathbb{F}\langle\mathcal{X}\rangle$ the free unital associative $\mathbb{F}$-algebra on $\mathcal{X}$, and $T$ a $T$-ideal of $\mathbb{F}\langle\mathcal{X}\rangle$.  Let $T(A)$ denote the two-sided ideal of $A$ consisting of all evaluations of polynomials from $T$ on $A$.~In particular, we should note that if $f\in\mathbb{F}\langle\mathcal{X}\rangle$ and $(f)^T$ denotes the $T$-ideal generated by $f$, then $(f)^T(A)$ is precisely the two-sided ideal of $A$ generated by all evaluations of $f$ on $A$.

Recall that a polynomial $p(x_{1},\ldots,x_{m})\in\mathbb{F}\langle\mathcal X\rangle$ is \emph{nonmatrix} if it is not a polynomial identity of $\mathrm{M}_{2}(\mathbb{F})$, i.e., the value $p(a_{1},\ldots,a_{m})\neq 0$ for some  $a_{1},\ldots,a_{m}\in\mathrm{M}_{2}(\mathbb{F})$. Following
\cite{BresarZelmanov2026}, a $T$-ideal $T$ is called a
\emph{nonmatrix commutator $T$-ideal}~if
$T\subseteq([x_1,x_2])^T$  and $T(\mathrm{M}_2(\mathbb{F}))\neq 0.$

\smallskip

We use the comparison result of \cite[Lemma~2.4]{BresarZelmanov2026}. Under the additional assumption that the relevant $T$-ideals are closed under partial linearizations, it states that for nonmatrix commutator $T$-ideals $T_1$ and $T_2$, the equalities $A=T_1(A)$ and $A=T_2(A)$ are equivalent for every unital algebra $A$. For a $T$-ideal generated by a multilinear polynomial, as in the case at hand, this closure property is a consequence of the following lemma.

\begin{lemma}\label{lem:partial-linearization}
Let $f\in\mathbb{F}\langle\mathcal X\rangle$ be multilinear. Then the $T$-ideal $(f)^T$ is multihomogeneous with respect to the standard multigrading by the degrees in the variables.  In particular,~$(f)^T$ is closed under partial linearizations.
\end{lemma}

    \begin{proof}
    It is straightforward to check that every element of $(f)^T$ is a finite sum of terms of the form $r\theta(f)s$, where $r,s\in\mathbb{F}\langle\mathcal X\rangle$ and $\theta$ is an endomorphism of the free algebra.~Writing each $\theta(x_i)$ as a sum of its multihomogeneous components, $\theta(f)$ expands as a sum of terms $f(g_1,\ldots,g_m)$, where each $g_i$ is a component of $\theta(x_i)$.~The map $x_i\mapsto g_i$ extends to an endomorphism $\theta'$ of the free algebra, so $f(g_1,\ldots,g_m)=\theta'(f)\in (f)^T$. Moreover, since $f(x_1,\ldots,x_m)$ is multilinear, $\theta'(f)$ has a well-defined multidegree. Hence if $r'$ and $s'$ are components of $r$ and $s$ in the standard multigrading, then $r'\theta'(f)s'$ has a single multidegree and lies in $(f)^T$.~It follows that every multigraded component of $r\theta(f)s$ lies in $(f)^T$, so $(f)^T$ itself is multihomogeneous. 

    Now take $h\in (f)^T$.~Replacing a variable by a sum of new variables leaves it in $(f)^T$.~Its partial linearizations are precisely the multihomogeneous components corresponding to the prescribed degrees in the new variables. Since $(f)^T$ is multihomogeneous, it contains these components and is closed under partial linearizations.
    \end{proof}

Let $p(x_1,\dots,x_m)$ be a multilinear polynomial and $A$ an arbitrary algebra, and define $L_p(A)=\operatorname{span}_{\mathbb{F}}\{p(a_1,\dots,a_m)\mid a_i\in A\}$. Multilinearity yields stability of $L_p(A)$ under inner derivations, as recorded below.

\begin{lemma}\label{lem:lie-values}
If $p$ is multilinear, then $L_p(A)$ is a Lie ideal of $A$.
\end{lemma}

    \begin{proof}
    For $a,a_{1},\dots,a_{m}\in A$, iterated application of $[a,uv]=[a,u]v+u[a,v]$ gives, first for multilinear monomials and then by linearity for $p$,
    $$[a,p(a_{1},\dots,a_{m})]=\sum_{i=1}^{m}p(a_{1},\dots,[a,a_{i}],\dots,a_{m}).$$
    Consequently, every commutator of an element of $A$ with a value of $p$ belongs to $L_{p}(A)$. Therefore, $[A,L_{p}(A)]\subseteq L_{p}(A)$, and $L_{p}(A)$ is a Lie ideal of $A$.
    \end{proof}

The following elementary lemma connects the preceding $T$-ideal comparison to generation of the ambient algebra.

\begin{lemma}\label{lem:lie-ideal-square}
Let $A$ be a unital algebra with a Lie ideal $L$. If $[L,L]$ generates $A$ as a two-sided ideal, then
$A=L+L^2.$
\end{lemma}

    \begin{proof}
    Put $M=[L,L]$. The subspace $M$ remains a Lie ideal of $A$. Indeed, take $\ell_1,\ell_2\in L$ and $a\in A$. By the Jacobi identity,
    $$[[\ell_1,\ell_2],a]=[\ell_1,[\ell_2,a]]+[[\ell_1,a],\ell_2]\in[L,L],$$
    hence $[M,A]\subseteq M$. We 
    observe that $MA\subseteq AM.$ To verify this, if $m\in M$ and $a\in A$, then $ma=am+[m,a],$
    where $[m,a]\in M\subseteq AM$. It follows that
    $$AMA=A(MA)\subseteq A(AM)=AM,$$
    while $AM\subseteq AMA$ because $A$ is unital. Hence $AMA=AM$.
    Let $a\in A$ and $\ell_1,\ell_2\in L$. Since $L$ is a Lie ideal of $A$, we have
    $$a[\ell_1,\ell_2]=[a\ell_1,\ell_2]-[a,\ell_2]\ell_1\in L+L^2,$$
    so by linearity $AM\subseteq L+L^2$. In addition, under the assumption that $M=[L,L]$ generates $A$ as a two-sided ideal, $A=AMA=AM\subseteq L+L^2$, hence $A=L+L^2$.
    This proves the lemma. 
    \end{proof}

An element $p\in\mathbb{F}\langle\mathcal X\rangle$ is called a \emph{Lie polynomial} if it lies in the free Lie subalgebra generated by $\mathcal X$ with bracket $[u,v]=uv-vu$. We refer to \cite[\S1.2]{Reutenauer1993} for more information on Lie polynomials.

We record here, for later use, the following auxiliary consequence of \cite[Lemma~2.4]{BresarZelmanov2026} in a form convenient for the arguments below. Its only role in the proof of \thmCref is to pass from $A=K(A)$ to generation by iterated commutators.

\begin{proposition}\label{prop:lie-polynomial-generation}
Let $A$ be a unital algebra satisfying $A=K(A)$. Let $p=p(x_1,\ldots,x_m)$ be a multilinear Lie polynomial in $m\ge2$ variables which is not an identity of $\mathrm{M}_2(\mathbb{F})$. Then $$A=L_p(A)+L_p(A)^2.$$
In particular, $A=\Alg(L_p(A))$.
\end{proposition}

    \begin{proof}
    Put $L=L_p(A).$
    By \Cref{lem:lie-values}, the multilinearity of $p$ implies that $L$
    is a Lie ideal of $A$. We claim that $[L,L]$ generates $A$ as a two-sided ideal, and then  \Cref{lem:lie-ideal-square} immediately implies the desired equalities.

    Let $x_1,\ldots,x_m,y_1,\ldots,y_m$ be pairwise distinct elements of $\mathcal X$, and write $\mathbf x=(x_1,\ldots,x_m)$, $\mathbf y=(y_1,\ldots,y_m)$.  Put $q(\mathbf x, \mathbf y)=[p(\mathbf{x}),p(\mathbf{y})]$.
    Multilinearity of $p$ and pairwise distinctness of the $x_i$ and $y_i$ imply that $q$ is multilinear.
    In addition, since the commutator is bilinear, it follows that $\operatorname{span}_{\mathbb{F}}q(A)=[L,L]$.

    The goal is to show $q$ is not identically zero on $\mathrm{M}_2(\mathbb{F})$. Put
    $R=\mathrm{M}_2(\mathbb{F}).$ Every Lie monomial of degree at least two is a commutator of two Lie monomials of positive degree. Consequently, every value of $p$ on $R$ lies in $[R,R]$.
    Therefore, $L_p(R)\subseteq [R,R] = \mathfrak{sl}_2(\mathbb{F}).$ The hypothesis that $p$ is not an identity of the ring $R$ now yields $L_p(R)\neq 0$.
    Moreover, \Cref{lem:lie-values} shows that $L_p(R)$ is a Lie ideal of $R$, and therefore of the simple Lie algebra $\mathfrak{sl}_2(\mathbb{F})$.  Consequently, $L_p(R)=\mathfrak{sl}_2(\mathbb{F})$, and since $\mathfrak{sl}_2(\mathbb{F})$ is nonabelian, it follows~that $$\operatorname{span}_{\mathbb{F}}q(R)=[L_p(R),L_p(R)]=\mathfrak{sl}_2(\mathbb{F})\neq0.$$ In particular, $q$ is not an identity of $\mathrm{M}_2(\mathbb{F})$.

    Take $z_1,z_2\in\mathcal X$ not occurring in $\mathbf x$ or $\mathbf y$, and put $c(z_1,z_2)=[z_1,z_2]$. Let $\theta$ be the endomorphism of $\mathbb{F}\langle\mathcal{X}\rangle$ defined by $\theta(z_1)=p(\mathbf{x})$, $\theta(z_2)=p(\mathbf{y})$, and $\theta(z)=z$ for the remaining free generators. Then $\theta(c)=q(\mathbf{x},\mathbf{y})$, so $q\in(c)^T$ and thus $(q)^T\subseteq(c)^T$. Since $q(\mathrm{M}_2(\mathbb{F}))\neq 0$, both $(q)^T$ and $(c)^T$ are nonmatrix commutator $T$-ideals. Moreover, $c$ and $q$ are multilinear, so by \Cref{lem:partial-linearization}, these $T$-ideals are closed under partial linearizations.

    Observe that $(c)^T(A)=K(A)=A.$ Applying \cite[Lemma~2.4]{BresarZelmanov2026} to the pair $(c)^T$ and $(q)^T$ gives $A=(q)^T(A)$. Recall from the beginning of this section that $(q)^T(A)$ is the two-sided ideal of $A$ generated by the evaluations of~$q$. Moreover, since $\operatorname{span}_{\mathbb{F}}q(A)=[L,L]$, it follows that $A$ coincides with the two-sided ideal generated by $[L,L]$, as claimed. The proof is now complete.
    \end{proof}

\begin{remark}\label{rem:lie-hypothesis}
The two hypotheses on $p$ in \Cref{prop:lie-polynomial-generation} play different roles and should not be conflated. Multilinearity makes $L_p(A)$ a Lie ideal, while $p$ being Lie gives $L_p(\mathrm{M}_2(\mathbb{F}))\subseteq\mathfrak{sl}_2(\mathbb{F})$. The proposition need not hold if the polynomial is not assumed to be Lie. For example, consider the polynomial
$$r=[x_1,x_2][x_3,x_4]+[x_3,x_4][x_1,x_2].$$
We verify directly, over every field of characteristic different from $2$, that $r$ is a nonmatrix, nonzero multilinear central polynomial of $\mathrm{M}_2(\mathbb{F})$. If $X,Y\in\mathrm{M}_2(\mathbb{F})$ have trace zero, then a direct computation yields $XY+YX=\operatorname{tr}(XY)I_2.$ Every commutator in $\mathrm{M}_2(\mathbb{F})$ has trace zero. Hence for $a_i\in\mathrm{M}_2(\mathbb{F})$, $1 \leq i\leq4$,
$$r(a_1,a_2,a_3,a_4)
=\operatorname{tr}\bigl([a_1,a_2][a_3,a_4]\bigr)I_2,$$
so every value of $r$ is central. On the other hand, let $E_{ij}$ denote the standard matrix units in $\mathrm{M}_2(\mathbb{F})$. Then $[E_{11},E_{12}]=E_{12}$ and $[E_{22},E_{21}]=E_{21}$, so
$$r(E_{11},E_{12},E_{22},E_{21})=E_{12}E_{21}+E_{21}E_{12}=I_2\neq 0.$$
Hence $r$ is nonmatrix with $L_r(\mathrm{M}_2(\mathbb{F}))=\mathbb{F}I_2$. It is not Lie, since any Lie polynomial of degree at least $2$ is $\mathfrak{sl}_2(\mathbb{F})$-valued on $\mathrm{M}_2(\mathbb{F})$, while $I_2\notin\mathfrak{sl}_2(\mathbb{F})$. Moreover, $K(\mathrm{M}_2(\mathbb{F}))=\mathrm{M}_2(\mathbb{F})$, as $\mathrm{M}_2(\mathbb{F})$ is simple as an associative algebra and has nonzero commutators. On the other hand, $$L_r(\mathrm{M}_2(\mathbb{F}))+L_r(\mathrm{M}_2(\mathbb{F}))^2=\mathbb{F}I_2\neq\mathrm{M}_2(\mathbb{F}),$$ so the conclusion of \Cref{prop:lie-polynomial-generation} does not extend to general non-Lie multilinear polynomials. 
\end{remark}

We now apply the preceding proposition to the Lie polynomial given by a double commutator.~As its values span $[[A,A],[A,A]]$, the resulting decomposition furnishes the connection required between the two hypotheses on $A$ discussed before.

\begin{proposition}\label{cor:double-generation}
Assume that $A$ is an unital algebra with $A=K(A)$. Put $G=[[A,A],[A,A]]$. Then $A=G+G^{2}$. In particular, $A=\Alg(G)$.
\end{proposition}

    \begin{proof}
    To apply \Cref{prop:lie-polynomial-generation}, put $p=[[x_1,x_2],[x_3,x_4]]$. This Lie polynomial has value set spanning $L_p(A)=[[A,A],[A,A]]=G$.
    It remains to show that $p$ is not an identity of $\mathrm{M}_2(\mathbb{F})$. Since $[E_{12},E_{21}]=E_{11}-E_{22}$ and $[E_{11},E_{12}]=E_{12}$, we have
$$p(E_{12},E_{21},E_{11},E_{12})=[E_{11}-E_{22},E_{12}]=2E_{12}\neq 0,$$
    so  $A=G+G^2$, and thus $A=\Alg(G)$.
    \end{proof}

The reverse implication needs no $T$-ideal argument because $A=K(A)$ is already a consequence of generation by double commutators in view of \Cref{lem:domain-structure}.~Hence a unital algebra is generated as a two-sided ideal by its commutators if and only if it is generated as an algebra by its double commutators.~Moreover, the stronger decomposition into a linear span obtained above remains valid.

\begin{corollary}\label{cor:unital-equivalence}
Let $A$ be a unital algebra and put $G=[[A,A],[A,A]]$.~Then the following statements are equivalent:
    \begin{enumerate}[label=\textup{(\arabic*)}]
    \item $A=K(A)$.
    \item $A=\Alg(G)$.
    \item $A=G+G^2$.
    \end{enumerate}
\end{corollary}

    \begin{proof}
    The implication from \textup{(1)} to \textup{(3)} follows from \Cref{cor:double-generation}, the passage from \textup{(3)} to \textup{(2)} is immediate, and the implication from \textup{(2)} to \textup{(1)} follows from \Cref{lem:domain-structure}.
    \end{proof}

\begin{lemma}[{\cite[Proposition~4.3(b)]{BresarZelmanovSurvey2025}}]\label{lem:unit-range}
Let $A$ be unital with $1_A$ and let $\varphi\colon A\to B$ be a Jordan homomorphism. Then $\varphi(1_A)$ is the unit of $\Alg(\varphi(A))$.
\end{lemma}

With the necessary preparations complete, we now prove \thmCref, stated in the Introduction as the main application.

\begin{proof}[Proof of \thmCref]
It follows from \Cref{lem:unit-range}  that $\Alg(\varphi(A))$ is unital with identity~$1$, and  hence $\Ann(\Alg(\varphi(A)))=0$. Moreover, by \Cref{cor:unital-equivalence}, we also have $A=\Alg([[A,A],[A,A]])$,  so \thmAref gives $V_{\varphi}\cap W_{\varphi}=0$. By \Cref{prop:standard-mod-Q}, $\varphi$ is splittable and standard.
Applying \thmAref again, we now obtain the direct sum decomposition 
$\Alg(\varphi(A))=V_{\varphi}\oplus W_{\varphi}$. Thus $1=e+f$ with $e\in V_{\varphi}$, $f\in W_{\varphi}$, $ef=fe=0$, so $e,f$ are idempotents. Writing $c=v+w$ with $v\in V_{\varphi}$ and $w\in W_{\varphi}$, we get $fc=w=cf$, so $f$ is central in $\Alg(\varphi(A))$.  \Cref{prop:standard-mod-Q} gives the decomposition $\varphi=\varphi_{+}+\varphi_{-}$, where $\varphi_{+},\varphi_{-}\colon A\to B$ are respectively the homomorphic and antihomomorphic parts valued in $W_{\varphi}$ and $V_{\varphi}$. Put $\varepsilon=f$. Then $\varepsilon\varphi=\varphi_{+}$ and $(1-\varepsilon)\varphi=\varphi_{-}$, which gives the desired (anti)multiplicative properties.

It remains to prove the uniqueness of $\varepsilon$.
Suppose that $\varepsilon'\in Z(\Alg(\varphi(A)))$ is
another central idempotent with the stated
properties. Because $\varepsilon'\varphi$ is a homomorphism and $(1-\varepsilon')\varphi$ is an antihomomorphism, we have $\varepsilon' a(x,y)=0$ and $(1-\varepsilon')b(x,y)=0$ for all $x,y\in A$. Centrality of $\varepsilon'$ then yields $\varepsilon' V_{\varphi}=0$ and $(1-\varepsilon')W_{\varphi}=0$, so $\varepsilon' e=0$ and $\varepsilon'\varepsilon=\varepsilon$. Consequently,
$\varepsilon' = \varepsilon'(e+\varepsilon) = \varepsilon,$
and uniqueness follows.
\end{proof}

Thus, \thmCref holds without the surjectivity hypothesis that was left open after \cite[Theorem~2.10]{BresarZelmanov2026}.

\end{document}